\documentclass{article}

\usepackage[ams]{maintemplate}
\usepackage{array}
\usepackage{tabularx}
\usepackage{chngcntr}

\newcolumntype{V}{>{\centering\arraybackslash} m{.2\linewidth} }

\usetikzlibrary{fit}

\newcommand{\ext}{\operatorname{ext}}

\makeatletter
\theoremstyle{plain}
\newtheorem*{rep@theorem}{\rep@title}
\newcommand{\newreptheorem}[2]{%
\newenvironment{rep#1}[1]{%
 \def\rep@title{#2 \ref*{##1}}%
 \begin{rep@theorem}}%
 {\end{rep@theorem}}}
\makeatother

\newreptheorem{thm}{Theorem}

\title{Turing degrees of hyperjumps}
\author{
	Hayden R. Jananthan \\ 
	Lincoln Laboratory Supercomputing Center \\ 
	Massachusetts Institute of Technology \\ 
	Lexington, MA 02421, USA \\ 
	\url{https://jananthan.mit.edu} \\ 
	\url{hayden.r.jananthan@gmail.edu} \\ 
	\\
	Stephen G. Simpson \\ 
	Department of Mathematics \\ 
	Vanderbilt University \\ 
	Nashville, TN 37203, USA \\ 
	\url{https://sgslogic.net} \\ 
	\url{sgslogic@gmail.com}
}
\date{
	First draft: August 25, 2019 \\ 
	This draft: July 18, 2024
}

\begin{document}

\maketitle

\begin{abstract}
	The Posner-Robinson Theorem states that for any reals $Z$ and $A$ such that $Z \oplus 0' \turingleq A$ and $0 \turingle Z$, there exists $B$ such that $A \turingeq B' \turingeq B \oplus Z \turingeq B \oplus 0'$. Consequently, any nonzero Turing degree $\turingdeg(Z)$ is a Turing jump relative to some $B$. Here we prove the hyperarithmetical analog, based on an unpublished proof of Slaman, namely that for any reals $Z$ and $A$ such that $Z \oplus \mathcal{O} \turingleq A$ and $0 \hyple Z$, there exists $B$ such that $A \turingeq \mathcal{O}^B \turingeq B \oplus Z \turingeq B \oplus \mathcal{O}$. As an analogous consequence, any nonhyperarithmetical Turing degree $\turingdeg(Z)$ is a hyperjump relative to some $B$. 
\end{abstract}

\tableofcontents

\section{Introduction}

Our starting point is the \emph{Friedberg Jump Theorem}:

\begin{thm}[Friedberg Jump Theorem] \label{friedberg jump theorem intro} \cite[Theorem 13.3.IX, pg. 265]{rogers1967theory}
	Suppose $A$ is a real such that $0' \turingleq A$. Then there exists $B$ such that
	\begin{equation*}
		A \turingeq B' \turingeq B \oplus 0'.
	\end{equation*}
\end{thm}

There are several refinements of the Friedberg Jump Theorem. One such extension shows that $B$ can be taken to be an element of any special $\Pi^0_1$ class $P \subseteq \cantor$. Here \emph{special} means that $P$ is nonempty and has no recursvie elements.

\begin{thm} \label{jump inversion in special pi01 class intro} \cite[following Theorem 3.1, pg. 37]{jockusch1971classes}
	Suppose $P \subseteq \cantor$ is a special $\Pi^0_1$ class and $A$ is a real such that $0' \turingleq A$. Then there exists $B \in P$ such that
	\begin{equation*}
		A \turingeq B' \turingeq B \oplus 0'.
	\end{equation*}
\end{thm}

Another refinement is the \emph{Posner-Robinson Theorem}:

\begin{thm}[Posner-Robinson Theorem] \label{posner-robinson theorem intro} \cite[Theorem 1, pg. 715]{posner1981degrees} \cite[Theorem 3.1, pg. 1228]{jockusch1984pseudojump}
	Suppose $Z$ and $A$ are reals such that $Z \oplus 0' \turingleq A$ and $0 \turingle Z$. Then there exists $B$ such that
	\begin{equation*}
		A \turingeq B' \turingeq B \oplus Z \turingeq B \oplus 0'.
	\end{equation*}
\end{thm}

In this paper we prove hyperarithmetical analogs of \cref{jump inversion in special pi01 class intro} and \cref{posner-robinson theorem intro}. The hyperarithmetical analog of \cref{friedberg jump theorem intro} is due to Macintyre \cite[Theorem 3, pg. 9]{macintyre1977transfinite}. In these hyperarithmetical analogs, the Turing jump operator $X \mapsto X'$ is replaced by the hyperjump operator $X \mapsto \mathcal{O}^X$ and $\Pi^0_1$ classes are replaced by $\Sigma^1_1$ classes. A feature of \cite[Theorem 3, pg. 9]{macintyre1977transfinite} and of our results is that they involve Turing degrees rather than hyperdegrees, so for instance $\mathcal{O}^B$ is not only hyperarithmetically equivalent to $A$, but in fact Turing equivalent to $A$. 

Here is an outline of this paper:

In \S 2 we prove the following basis theorem for uncountable $\Sigma^1_1$ classes $K \subseteq \cantor$.

\begin{repthm}{gandy-kreisel basis theorem}
	Suppose $K \subseteq \cantor$ is an uncountable $\Sigma^1_1$ class and $Z$ and $A$ are reals such that $Z \oplus \mathcal{O} \turingleq A$ and $0 \hyple Z$. Then there exists $B \in K$ such that 
	\begin{equation*}
		A \turingeq \mathcal{O}^B \turingeq B \oplus \mathcal{O}
	\end{equation*}
	and $Z \hypnleq B$.
\end{repthm}

In \S 3 we prove the following analog of \cref{posner-robinson theorem intro}, which is essentially due to Slaman \cite{slaman2018email}.

\begin{repthm}{posner-robinson for hyperjumps}
	Suppose $Z$ and $A$ are reals such that $Z \oplus \mathcal{O} \turingleq A$ and $0 \hyple Z$. Then there exists $B$ such that
	\begin{equation*}
		A \turingeq \mathcal{O}^B \turingeq B \oplus Z \turingeq B \oplus \mathcal{O}.
	\end{equation*}
\end{repthm}

The remainder of this section fixes notation and terminology.

$g \colonsub A \to B$ denotes a partial function with domain $\dom g \subseteq A$ and codomain $B$. For $a\in A$, if $a \in \dom g$ then we say `$g(a)$ converges' or `$g(a)$ is defined' and write $g(a) \converge$. Otherwise, we say `$g(a)$ diverges' or `$g(a)$ is undefined' and write $g(a) \diverge$. If $f$ and $g$ are two partial functions ${\subseteq} A \to B$ and $a\in A$, then $f(a) \simeq g(a)$ means $( f(a) \converge \wedge g(a) \converge \wedge f(a)=g(a)) \vee (f(a) \diverge \wedge g(a) \diverge)$. We write $f(a) \converge = b$ to mean that $f(a) \converge$ and $f\colon a \mapsto b$.

$\baire$ and $\cantor$ denote the Baire and Cantor spaces, respectively, whose elements we sometimes call \emph{reals}. We identify $\cantor$ and the powerset $\mathcal{P}(\mathbb{N})$ in the usual manner.

If $S$ is a set, then $S^\ast$ is the set of strings of elements from $S$. If $s_0,\ldots,s_{n-1} \in S$, then $\sigma = \langle s_0,\ldots,s_{n-1}\rangle \in S^\ast$ denotes the string of \emph{length $|\sigma| \coloneq n$} defined by $\sigma(k) = s_k$. If $\langle s_0,\ldots,s_{n-1}\rangle, \langle t_0,\ldots, t_{m-1}\rangle \in S^\ast$, then their \emph{concatenation} is $\langle s_0,\ldots, s_{n-1}\rangle \concat \langle t_0,\ldots, t_{m-1}\rangle \coloneq \langle s_0,\ldots, s_{n-1}, t_0, \ldots, t_{m-1}\rangle$. If $\sigma, \tau \in S^\ast$, then $\sigma$ is an \emph{initial segment} of $\tau$ (equivalently, $\tau$ is an extension of $\sigma$) written $\sigma \subseteq \tau$, if $\tau \restrict |\sigma| = \sigma$. If $f \colon \mathbb{N} \to S$ then $\sigma \in S^\ast$ is an \emph{initial segment} of $f$ (equivalently, $f$ is an \emph{extension} of $\sigma$), written $\sigma \subset f$, if $f \restrict |\sigma| = \sigma$. $\sigma, \tau \in S^\ast$ are \emph{incompatible} if neither is an initial segment of the other. If $\leq$ is a partial order on $S$, then the \emph{lexicographical ordering} $\leq_\mathrm{lex}$ on $S^\ast$ is defined by setting $\sigma \leq_\mathrm{lex} \tau$ if $\sigma \subseteq \tau$ or, where $k$ is the least index at which $\sigma(k) \neq \tau(k)$, then $\sigma(k) < \tau(k)$.

$\varphi_e^{(k)}$ denotes the $e$-th partial recursive function ${\subseteq} \mathbb{N}^k \to \mathbb{N}$; $e$ is called an \emph{index} of $\varphi_e^{(k)}$. Likewise, if $f \in \baire$ then $\varphi_e^{(k),f}$ denotes the $e$-th partial function $\varphi_e^{(k),f}\colon {\subseteq} \mathbb{N}^k \to \mathbb{N}$ which is partial recursive in $f$; $e$ is again called an \emph{index} of $\varphi_e^{(k),f}$, while $f$ is called an \emph{oracle} of $\varphi_e^{(k),f}$. 

$\turingleq$ denotes Turing reducibility while $\turingeq$ denotes Turing equivalence. $\hypleq$ denotes hyperarithmetical reducibility while $\hypeq$ denots hyperarithmetical equivalence. For $X \in \cantor$, $X'$ denotes the Turing jump of $X$ and $\mathcal{O}^X$ denotes the hyperjump of $X$. $\mathcal{O}$ denotes Kleene's $\mathcal{O}$. For $f,g \in \baire$, their \emph{join} $f \oplus g \in \baire$ is defined by $(f \oplus g)(2n) = f(n)$ and $(f\oplus g)(2n+1) = g(n)$.

$P_e$ denotes the $e$-th $\Pi^0_1$ set $\{ f \in \baire \mid \varphi_e^{(1),f}(0) \converge\} \subseteq \baire$. $P_e^\ast$ denotes the $e$-th $\Sigma^1_1$ class $\{ X \in \cantor \mid \exists f~ (f \oplus X \in P_e) \}$.

\section{A Basis Theorem for \texorpdfstring{$\Sigma^1_1$}{Sigma11} Classes}

The following theorem includes the Gandy Basis Theorem \cite[Theorem III.1.4, pg. 54]{sacks1990higher}, the Kreisel Basis Theorem for $\Sigma^1_1$ Classes \cite[Theorem III.7.2, pg. 75]{sacks1990higher}, and Macintyre's Hyperjump Inversion Theorem \cite[Theorem 3, pg. 9]{macintyre1977transfinite}. 

\begin{thm} \label{gandy-kreisel basis theorem}
	Suppose $K \subseteq \cantor$ is an uncountable $\Sigma^1_1$ class and $Z$ and $A$ are reals such that $Z \oplus \mathcal{O} \turingleq A$ and $0 \hyple Z$. Then there exists $B \in K$ such that 
	\begin{equation*}
		A \turingeq \mathcal{O}^B \turingeq B \oplus \mathcal{O}
	\end{equation*}
	and $Z \hypnleq B$.
\end{thm}

To prove \cref{gandy-kreisel basis theorem} we use Gandy-Harrington forcing (first introduced by Harrington in an unpublished manuscript \cite{harrington1976powerless}; see, e.g., \cite[Theorem IV.6.3, pg. 108]{sacks1990higher}), forming a descending sequence of uncountable $\Sigma^1_1$ classes
\begin{equation*}
	K = K_0 \supseteq K_1 \supseteq \cdots \supseteq K_n \supseteq \cdots 
\end{equation*}
where an element of the intersection $\bigcap_{n=0}^\infty{K_n}$ has the desired property.  Unlike in the case of $\Pi^0_1$ subsets of $\cantor$, compactness cannot be used to easily show that the intersection $\bigcap_{n=0}^\infty{K_n}$ is nonempty. Instead, some care must be taken to show that this is the case.

\begin{prop} \label{basic facts about sigma11 sets}  
	\mbox{}
	\begin{enumerate}[(a)]
		\item Given a $\Sigma^1_1$ predicate $K \subseteq \cantor \times \mathbb{N}^k$, there is a primitive recursive function $f\colon \mathbb{N}^k \to \mathbb{N}$ such that
		\begin{equation*}
			P_{f(x_1,\ldots,x_k)}^\ast(X) \equiv K(X,x_1,\ldots,x_k).
		\end{equation*}
		\item Suppose $X \in \cantor$. Then $\{ e \in \mathbb{N} \mid X \notin P_e^\ast\} \turingeq \mathcal{O}^X$.
		
		\item $\{ e \in \mathbb{N} \mid P_e^\ast = \emptyset\} \turingeq \mathcal{O}$.
	\end{enumerate}
\end{prop}
\begin{proof}
	Straight-forward.
\end{proof}

\begin{cor}
	There exist primitive recursive functions $v$, $u$, and $U$ such that that for all $n,m \in \mathbb{N}$ and $\sigma,\tau \in N^\ast$ and $I \in \mathcal{P}_\mathrm{fin}(\mathbb{N})$,
	\begin{align*}
		P^\ast_{v(n,m)} & = P_n^\ast \cap P_m^\ast, \\
		P_{u(e,\sigma,\tau)}^\ast = P^\ast_e[\sigma,\tau]  & = \{ X \in \cantor \mid \sigma \subset X \wedge \exists g ~( X \oplus g \in P_e \wedge \tau \subset g) \}, \\
		P^\ast_{U(I,\sigma,\langle \tau_0,\ldots, \tau_{n-1}\rangle)} & = \bigcap_{k \in I \wedge k < n}{P_k^\ast[\sigma,\tau_k]}.
	\end{align*}
\end{cor}

\begin{prop} \label{extending solution and witnesses} \label{extending solutions and witnesses} \label{incompatible solutions}
	The following partial functions are $\mathcal{O}$-recursive:
	\begin{enumerate}[(a)]
		\item The partial function $\rho(\sigma,e) \simeq \langle \sigma_0,\sigma_1\rangle$ where $\sigma_0,\sigma_1$ are minimal incompatible extensions of $\sigma$ which have extensions in $P_e^\ast$ and $\sigma_0$ is lexicographically less than $\sigma_1$, whenever $\sigma$ has at least two extensions in $P_e^\ast$, otherwise diverging.
		\item The partial function $\ext(\langle e_1,\ldots, e_N\rangle, \sigma, \langle \tau_1, \ldots, \tau_N\rangle) \simeq (\tilde{\sigma}, \langle \tilde{\tau_1}, \ldots, \tilde{\tau_N}\rangle)$ where $(\tilde{\sigma}, \langle \tilde{\tau},\ldots,\tilde{\tau_N}\rangle)$ is the lexicographically least pair such that 
		\begin{enumerate}[1.]
			\item $\sigma \subsetneq \tilde{\sigma}$ and $\tau_k \subsetneq \tilde{\tau_k}$ for $1 \leq k \leq N$ and
			\item $\bigcap_{k=1}^N{P_{e_k}^\ast[\tilde{\sigma}, \tilde{\tau_k}]} \neq \emptyset$
		\end{enumerate}
		whenever $\bigcap_{k=1}^N{P_{e_k}^\ast[\sigma, \tau_k]} \neq \emptyset$, otherwise diverging.
	\end{enumerate}
\end{prop}
\begin{proof} \mbox{}
	\begin{enumerate}[(a)]
		\item Using $\mathcal{O}$, search for the first string $\nu$ such that $P_e^\ast[\sigma \concat \nu \concat \langle i \rangle, \langle\rangle] \neq \emptyset$ for $i=0,1$. Once such $\nu$ has been found, $\rho(\sigma,e) \converge = \langle \sigma \concat \nu \concat \langle 0\rangle, \sigma \concat \nu \concat \langle 1 \rangle \rangle$.
		
		\item Using $\mathcal{O}$, search for the first of $i=0,1$ for which $\bigcap_{k=1}^N{P_{e_k}^\ast[\sigma \concat \langle i\rangle, \tau_k]} \neq \emptyset$, then search for the lexicographically least $\langle j_1,\ldots, j_N\rangle \in \{0,1\}^N$ such that $\bigcap_{k=1}^N{P_{e_k}^\ast[\sigma \concat \langle i\rangle, \tau_k \concat \langle j_k\rangle]} \neq \emptyset$. If no such $i$ or $j_1,\ldots, j_N$ are found, then diverge. Otherwise, $\ext(\langle e_1,\ldots, e_N\rangle, \sigma, \langle \tau_1,\ldots, \tau_N\rangle) \converge = (\sigma \concat \langle i\rangle, \langle \tau_1 \concat \langle j_1\rangle, \ldots, \tau_N \concat \langle j_N\rangle)$.
	\end{enumerate}
\end{proof}

Let $\rho_0,\rho_1$ be defined by 
\begin{equation*}
	\rho(\sigma,e) \simeq \langle \rho_0(\sigma,e), \rho_1(\sigma,e)\rangle.
\end{equation*}

We use the ordinal notation description of $\mathcal{O}$ (and, more generally, $\mathcal{O}^Y$ for $Y \in \cantor$) described in \cite{sacks1990higher} and use the following well-known lemma to describe hyperarithmetical reducibility in terms of $H$-sets.

\begin{notation*}
	For $X \in \cantor$ and $n \in \mathbb{N}$, define
	\begin{equation*}
		(X)_n \coloneq \{ x \in \mathbb{N} \mid 2^n \cdot 3^x \in X\}.
	\end{equation*}
\end{notation*}

\begin{lem} \label{hyperarithmetical reducibility in terms of cross-sections}
	Suppose $X$ and $Y$ are reals in $\cantor$. Then $X \hypleq Y$ if and only if there exists $b \in \mathcal{O}^Y$ and $n \in \mathbb{N}$ such that $X = (H_b^Y)_n$.
\end{lem}
\begin{proof}
	Suppose $X \hypleq Y$, so that there is $b \in \mathcal{O}^Y$ such that $X \turingleq H_b^Y$. Let $e$ be the index of such a Turing reduction, i.e., let $e$ be such that $X = \varphi_e^{(1),H_b^Y}$. By definition \cite{sacks1990higher}, $2^b \in \mathcal{O}^Y$ and 
	\begin{equation*}
		H_{2^b}^Y \coloneq \{ 2^n3^x \mid \varphi_n^{(1),H_b^Y}(x) \converge\}.
	\end{equation*}
	Let $f$ be an index such that
	\begin{equation*}
		\varphi_f^{(1),H_b^Y}(x) \converge \iff \varphi_e^{(1),H_b^Y}(x) \converge = 1
	\end{equation*}
	Then
	\begin{align*}
		(H_{2^b}^Y)_f & = \{ x \in \mathbb{N} \mid \varphi_f^{(1),H_b^Y}(x) \converge\} \\
		& = \{ x \in \mathbb{N} \mid \varphi_e^{(1),H_b^Y}(x) \converge = 1 \} \\
		& = X
	\end{align*}
	
	Conversely, suppose there is $b \in \mathcal{O}^Y$ and $n \in\mathbb{N}$ such that $X = (H_b^Y)_n$. Let $e$ be an index such that
	\begin{equation*}
		\varphi_e^{(1),Z}(x) = \begin{cases} 1 & \text{if $x\in (Z)_n$} \\ 0 & \text{if $x \notin (Z)_n$} \end{cases}
	\end{equation*}
	for any $Z \in \cantor$. Then $\varphi_e^{(1),H_b^Y} = X$, showing that $X \turingleq H_b^Y$.
\end{proof}

\begin{proof}[Proof of \cref{gandy-kreisel basis theorem}.]
	By the Gandy Basis Theorem \cite[Theorem III.1.4, pg. 54]{sacks1990higher}, assume without loss of generality that $\omega_1^Y = \omega_1^\ck$ for all $Y \in K$. 
	
	In order to control the hyperjump $\mathcal{O}^B$, we choose $B$ to be an element of an intersection of $\Sigma^1_1$ subsets
	\begin{equation*}
		K = K_0 \supseteq K_1 \supseteq \cdots \supseteq K_n \supseteq \cdots.
	\end{equation*}
	
	In order for $B$ to be an element of $K_n = P_{j(n)}^\ast$ for each $n$, there must be $g_n \in \baire$ such that $B \oplus g_n \in P_{j(n)}$, where $j(n)$ is some index of $K_n$. Such $g_n$ depend on $B$. Thus, we additionally define sequences of strings
	\begin{equation*}
		\begin{array}{ccccccccc}
			\sigma_0 & \subseteq & \sigma_1 & \subseteq & \cdots & \subseteq & \sigma_n & \subseteq & \cdots \\
			\tau_{0,0} & \subseteq & \tau_{1,0} & \subseteq & \cdots & \subseteq & \tau_{n,0} & \subseteq & \cdots \\
			\tau_{0,1} & \subseteq & \tau_{1,1} & \subseteq & \cdots & \subseteq & \tau_{n,0} & \subseteq & \cdots \\
			\tau_{0,2} & \subseteq & \tau_{1,2} & \subseteq & \cdots & \subseteq & \tau_{n,0} & \subseteq & \cdots \\
			\vdots & & \vdots & & \ddots & & \vdots & & \ddots
		\end{array}
	\end{equation*}
	so that $B = \bigcup_{n\in \omega}{\sigma_n}$ and $g_k = \bigcup_{n\in\omega}{\tau_{n,k}}$. 
	We also define a sequence of finite subsets of $\mathbb{N}$
	\begin{equation*}
		I_0 \subseteq I_1 \subseteq \cdots \subseteq I_n \subseteq \cdots
	\end{equation*}
	encoded as finite sequences $\{e_1,\ldots,e_N\} \mapsto \langle e_1,\ldots,e_N\rangle$ which keep track of the indices $e$ of $\Sigma^1_1$ classes we have committed to intersecting, so that $K_n = \bigcap_{k\in I_n}{P_k^\ast[\sigma_n,\tau_{n,k}]}$. 
	A function $j\colon \mathbb{N} \to \mathbb{N}$ keeps track of the index of $K_n$, i.e., 
	\begin{equation*}
		K_n = P_{j(n)}^\ast.
	\end{equation*}
	In the course of the proof, we assume that $j$ encodes all of the information from previous steps (i.e., a course-of-value computation) though we avoid making this precise to ease the burden of notation. 
	
	To ease in the notation and exposition, we set the following temporary definitions. An \emph{intersection system} consists of the following data:
	\begin{enumerate}[(i)]
		\item a finite subset $I \subseteq \mathbb{N}$, 
		\item a string $\sigma$, and
		\item a sequence of strings $\langle \tau_k \mid k \in I\rangle$
	\end{enumerate}
	subject to the constraint that $\bigcap_{k \in I}{P_k^\ast[\sigma,\tau_k]}$ is nonempty. If $k \notin I$, then we assign the value $\langle\rangle$ to $\tau_k$.
	
	By \emph{adding $P_e^\ast$ to the intersection system $I,\sigma,\langle \tau_k \mid k \in I\rangle$}, we mean the following procedure, where $K = \bigcap_{k \in I}{P_k^\ast[\sigma,\tau_k]}$:
	\begin{description}
		\item[Case $1$: $K \cap P_e^\ast = \emptyset$.] Let $\tilde{I} = I$, $\tilde{K} = K$, $\tilde{\sigma} = \sigma$, and $\tilde{\tau_k} = \tau_k$ for each $k$.
		\item[Case $2$: $K \cap P_e^\ast \neq \emptyset$.] Let $\tilde{I} = I \cup \{e\}$, and let $\tilde{\sigma}$ and, simultaneously for all $k \in \tilde{I}$, $\tilde{\tau_k}$ be the lexicographically least proper extensions of $\sigma$ and $\tau_k$, respectively, such that $\bigcap_{k \in \tilde{I}}{P_k^\ast[\tilde{\sigma},\tilde{\tau_k}]} \neq \emptyset$.
	\end{description}
	The resulting intersection system is $\tilde{I},\tilde{\sigma}, \langle \tilde{\tau_k} \mid k \in \tilde{I}\rangle$. Note that from $I,\sigma,\langle \tau_k \mid k \in I\rangle$ and $e$, the new intersection system $\tilde{I},\tilde{\sigma},\langle \tilde{\tau_k} \mid k \in \tilde{I}\rangle$ can be determined in a uniform way recursively in $\mathcal{O}$: representing $I$ as $\langle e_1,\ldots, e_N\rangle$ and writing $e_{N+1} = e$, then 
	\begin{align*}
		\tilde{I} & = \begin{cases} \langle e_1,\ldots, e_N, e_{N+1}\rangle & \text{if $K \cap P_e^\ast \neq \emptyset$,} \\ I & \text{otherwise,} \end{cases} \\
		(\tilde{\sigma}, \langle \tilde{\tau_k} \mid k \in \tilde{I}\rangle) & = 
		\begin{cases} 
			\ext(\tilde{I}, \sigma, \langle \tau_{e_1},\ldots, \tau_{e_N}, \langle\rangle\rangle) & \text{if $K \cap P_e^\ast \neq \emptyset$,} \\ 
			(\sigma, \langle \tau_k \mid k \in I\rangle) & \text{otherwise.} 
		\end{cases}
	\end{align*}
	In particular, the index $U(\tilde{I},\tilde{\sigma}, \langle \tilde{\tau_k} \mid k < \max I\rangle)$ of $\tilde{K}$ can be determined uniformly from the intersection system $I, \sigma, \langle \tau_k \mid k \in I\rangle$ using $\mathcal{O}$ as an oracle.
	
	Now we proceed with the construction. As $K$ is $\Sigma^1_1$, there is $e_0$ such that $K = P_{e_0}^\ast$. 
	\begin{description}
		\item[Stage $n=0$:] Define
		\begin{equation*}
			K_0 \coloneq K, \qquad \sigma_0 \coloneq \langle\rangle, \qquad \tau_{0,k} \coloneq \langle\rangle, \qquad j(0) \coloneq e_0, \qquad I_0 \coloneq \{e_0\}.
		\end{equation*}
		Note that $P_{j(0)}^\ast = K_0 = \bigcap_{k\in I_0}{P_k^\ast[\sigma_0,\tau_{0,k}]}$.
		
		\item[Stage $n= 3e+1$:] Let $I_n, \sigma_n, \langle \tau_{n,k} \mid k \in I_n\rangle$ be the result of adding $P_e^\ast$ to the intersection system $I_{n-1}, \sigma_{n-1}, \langle \tau_{n-1,k} \mid k \in I_{n-1}\rangle$, and let $K_n \coloneq \bigcap_{k\in I_n}{P_k^\ast[\sigma_n,\tau_{n,k}]}$ and $j(n)$ be an index for $K_n$.
		
		\item[Stage $n=3e+2$:] At this stage we encode $A(e)$ into $B$. 
		
		By construction, 
		\begin{equation*}
			P_{j(n-1)}^\ast = K_{n-1} = \bigcap_{k \in I_{n-1}}{P_k^\ast[\sigma_{n-1},\tau_{n-1,k}]} \neq \emptyset.
		\end{equation*}
		As $K_{n-1}$ is uncountable, there are infinitely many pairwise-incompatible extensions of $\sigma_{n-1}$ which extend to elements of $K_{n-1}$. Thus, let
		\begin{equation*}
			\sigma_n \coloneq \rho_{A(e)}(\sigma_{n-1},j(n-1)).
		\end{equation*}
		Define
		\begin{align*}
			K_n & \coloneq \bigcap_{k \in I_{n-1}}{P_k^\ast[\sigma_n,\tau_{n-1,k}]} = P_{U(\sigma_n,I_{n-1},\langle \tau_{n-1,0},\ldots,\tau_{n-1,n-1}\rangle)}, \\
			\tau_{n,k} & \coloneq \tau_{n-1,k}, \qquad (\text{for all $k$}) \\
			I_n & \coloneq I_{n-1}, \\
			j(n) & \coloneq U(\sigma_n,I_{n-1},\langle \tau_{n-1,0},\ldots,\tau_{n-1,n-1}\rangle).
		\end{align*}
		
		\item[Stage $n = 3^{b+1} \cdot 5^e \cdot 7^f$:] Suppose $b \in \mathcal{O}$. Let $m \in \mathbb{N}$ be the least natural number for which there are $Y_1,Y_2 \in K_{n-1}$ such that $\varphi_{f}^{(1),H_b^{Y_1}}(2^e\cdot 3^m)$ and $\varphi_{f}^{(1),H_b^{Y_2}}(2^e \cdot 3^m)$ are both defined and unequal. For $i \in \{0,1\}$, let
		\begin{equation*}
			K_{n-1}^i = \{ Y \in K_{n-1} \mid \varphi_{f}^{(1),H_b^{Y_1}}(2^e \cdot 3^m) \converge = i\}.
		\end{equation*}
		Because $K^0_{n-1} \cap K^1_{n-1} = \emptyset$, there is a least $k \in \mathbb{N}$ and $i \in \{0,1\}$ such that $\{ Y \in K_{n-1}^0 \mid Y(k) = i \}$ and $\{ Y \in K_{n-1}^1 \mid Y(k) \neq i\}$ are nonempty. Let $i_0 = i$ and $i_1 = 1-i$.
		
		Let $I_n,\sigma_n,\langle \tau_{n,k} \mid k \in I_n\rangle$ be the result of adding the (uniformly in $b$, $e$, $f$, $m$, $k$,and $i$, given $Z(m)$) $\Sigma^1_1$ class $\{ Y \in \cantor \mid \varphi_{f}^{(1),H_b^Y}(2^e \cdot 3^m)\converge \neq Z(m) \wedge Y(k) \neq i_{Z(m)}\}$ to the intersection system $I_{n-1},\sigma_{n-1},\langle \tau_{n-1,k} \mid k \in I_{n-1}\rangle$, and let $K_n \coloneq \bigcap_{k\in I_n}{P_k^\ast[\sigma_n,\tau_{n,k}]}$ and $j(n)$ be an index for $K_n$. 
		
		If $b \notin \mathcal{O}$ or no such $m$ exists, do nothing, i.e., let 
		\begin{equation*}
			K_n \coloneq K_{n-1},\qquad \sigma_n \coloneq \sigma_{n-1}, \qquad \tau_{n,k} \coloneq \tau_{n-1,k}, \qquad j(n) \coloneq j(n-1), \qquad I_n \coloneq I_{n-1}.
		\end{equation*}
		
		\item[All Other Stages $n$:] Do nothing, i.e., let
		\begin{equation*}
			K_n \coloneq K_{n-1},\qquad \sigma_n \coloneq \sigma_{n-1}, \qquad \tau_{n,k} \coloneq \tau_{n-1,k}, \qquad j(n) \coloneq j(n-1), \qquad I_n \coloneq I_{n-1}.
		\end{equation*}
	
	\end{description}
	
	This completes the construction.
	
	Define
	\begin{equation*}
		B \coloneq \bigcup_{n\in\mathbb{N}}{\sigma_n} \quad \text{and} \quad g_k \coloneq \bigcup_{n\in\mathbb{N}}{\tau_{n,k}}.
	\end{equation*}
	We start by claiming $B \in \bigcap_{n\in\mathbb{N}}{K_n}$: by construction, for $k \in \bigcap_{n\in\mathbb{N}}{I_n}$, we have $B \oplus g_k \in P_k$, showing $B \in P_k^\ast$. Additionally, by construction $B \in P_k^\ast[\sigma_n,\tau_{n,k}]$ for every $n$ and every $k \in \bigcap_{n\in\mathbb{N}}{I_n}$, so $B \in \bigcap_{k\in I_n}{P_k^\ast[\sigma_n,\tau_{n,k}]} = K_n$. Thus, $B \in \bigcap_{n\in\mathbb{N}}{K_n}$. In particular, $B \in K_0 = K$, so $\omega_1^B = \omega_1^\ck$.
	
	If $Z \hypleq B$, then \cref{hyperarithmetical reducibility in terms of cross-sections} shows there are $c \in \mathcal{O}^B$ and $e \in \mathbb{N}$ such that $Z = (H_b^B)_e$. Because $\omega_1^B = \omega_1^\ck$, there exists $b \in \mathcal{O}$ such that $|b|=|c|$ and hence $H_b^B \turingeq H_c^B$ by Spector's Uniqueness Theorem \cite[Corollary II.4.6, pg. 40]{sacks1990higher}. Let $f$ be an index such that $\varphi_f^{(1),H_b^B} = H_c^B$, so that $Z = (\varphi_f^{(1),H_b^B})_e$. By construction, at Stage $n=3^{b+1}\cdot 5^e \cdot 7^f$ it must have been the case that no $m$ and $Y_1,Y_2 \in K_{n-1}$ existed with $\varphi_f^{(1),H_b^{Y_1}}(2^e\cdot 3^m)$ and $\varphi_f^{(1),H_b^{Y_2}}(2^e\cdot 3^m)$ both defined and unequal. In particular, $\varphi_f^{(1),H_b^B}$ is a $\Sigma^1_1$ singleton, and so hyperarithmetical. But then $H_c^B \turingeq H_b^B$ is hyperarithmetical, hence $Z = (H_c^B)_e$ is hyperarithmetical, a contradiction. Thus, $Z \nhypleq B$.
	
	We now make the following observations: assuming $j(n-1)$ is known (and utilizing the implicit course-of-values procedure to yield $I_{n-1}, \sigma_{n-1}, \langle \tau_{n-1,k}\rangle_{k\in\mathbb{N}}$), then\ldots
	\begin{description}
		\item \ldots in Stage $n=3e+1$, the determination of $I_{n},\sigma_{n}, \langle \tau_{n,k}\rangle_{k\in\mathbb{N}}$ (and hence also $j(n)$) is recursive in $\mathcal{O}$ by \cref{extending solutions and witnesses}.
		
		\item \ldots in Stage $n=3e+2$, the determination of $I_n, \sigma_n, \langle \tau_{n,k}\rangle_{k\in\mathbb{N}}$ (and hence also $j(n)$) is recursive in $A$ (by construction) or $B \oplus \mathcal{O}$ (by determining the unique $i$ for which $\rho_i(\sigma_{n-1},j(n-1)) \subset B$) by \cref{incompatible solutions}.
		
		\item \ldots in Stage $n=3^{b+1}\cdot 5^e \cdot 7^f$, the determination of $I_n, \sigma_n, \langle \tau_{n,k}\rangle_{k\in\mathbb{N}}$ (and hence also $j(n)$) is recursive in $B \oplus \mathcal{O}$ (the determination of whether $b \in \mathcal{O}$ and whether there exists an $m$ and $Y_1, Y_2 \in K_{n-1}$ for which $\varphi_f^{(1),H_b^{Y_1}}(2^e\cdot 3^m)$ and $\varphi_f^{(1),H_b^{Y_2}}(2^e\cdot 3^m)$ are both defined and unequal may be performed recursively in $\mathcal{O}$ since it corresponds to checking whether a particular $\Sigma^1_1$ class is nonempty, and once the least such $m$ is found, we may determine the least $k$ and $i \in \{0,1\}$ for which $\{ Y \in K_{n-1}^0 \mid Y(k) = i\}$ and $\{ Y \in K_{n-1}^1 \mid Y(k) = 1-i\}$ are nonempty; finally, checking whether $B(k) = i$ or $B(k) = 1-i$ determines whether we intersected $\{ Y \in \cantor \mid \varphi_f^{(1),H_b^Y}(2^e\cdot 3^m) \downarrow = 0 \wedge Y(k) = i\}$ or $\{ Y \in \cantor \mid \varphi_f^{(1),H_b^Y}(2^e\cdot 3^m) \downarrow = 1 \wedge Y(k) = 1-i\}$, respectively) or $A$ (as before, the determination of whether $b \in \mathcal{O}$ and of the existence of such an $m$ may be done recursively in $\mathcal{O} \turingleq A$, and $Z \turingleq A$).
		
		\item \ldots in all other Stages $n$, the determination of $I_n,\sigma_n, \langle \tau_{n,k}\rangle_{k\in\mathbb{N}}$ (and hence also $j(n)$) is recursive.
	\end{description}
	In particular, $j \turingleq A$ and $j \turingleq B \oplus \mathcal{O}$.
	
	We make the following final observations:
	\begin{itemize}
		\item $A \turingleq j \oplus \mathcal{O}$ as $A(e)=i$ if and only if $j(n) = U(\rho_i(\sigma_{n-1},j(n-1)),I_{n-1},\langle \tau_{n-1,0},\ldots,\tau_{n-1,n-1}\rangle)$, where $n=3e+2$.
		
		\item $\mathcal{O}^B \turingleq j \oplus \mathcal{O}$ as $B \in P_e^\ast$ if and only if $v(j(n-1),e) \notin \{i \mid P_i^\ast =\emptyset \} \turingeq \mathcal{O}$. The determination $v(j(n-1),e) \notin \{i \mid P_i^\ast =\emptyset \} \turingeq \mathcal{O}$ can be made recursively in $j \oplus \mathcal{O}$.		
	\end{itemize}
	Thus, we find that 
	\begin{equation*}
		A \turingleq j \oplus \mathcal{O} \turingleq B \oplus \mathcal{O} \turingleq \mathcal{O}^B \turingleq j \oplus \mathcal{O} \turingleq A
	\end{equation*}
	so we have Turing equivalence throughout.
\end{proof}

The following corollary is originally due to Macintyre \cite[Theorem 3, pg. 9]{macintyre1977transfinite}.

\begin{cor}
	Suppose $A$ is a real such that $\mathcal{O} \turingleq A$. Then there exists $B$ such that
	\begin{equation*}
		A \turingeq \mathcal{O}^B \turingeq B \oplus \mathcal{O}.
	\end{equation*}
\end{cor}

The following corollary is ``folklore'', being unpublished but known to researchers and stated in \cite[Exercise 2.5.6, pg. 40]{chong2015recursion} without proof or references. Other than \cite[Exercise 2.5.6, pg. 40]{chong2015recursion} we have not seen any statement of \cref{gandy-kreisel basis theorem corollary} in the literature.

\begin{cor} \label{gandy-kreisel basis theorem corollary}
	Suppose $K$ is a nonempty $\Sigma^1_1$ class. Then there exists $B \in K$ such that $\mathcal{O} \turingeq \mathcal{O}^B \turingeq B \oplus \mathcal{O}$.
\end{cor}
\begin{proof}
	If $K$ is uncountable, then we apply \cref{gandy-kreisel basis theorem} with $Z = A = \mathcal{O}$. 
	
	If $K$ is countable, then its elements are hyperarithmetical \cite[Theorem III.6.2, pg. 72]{sacks1990higher} and so any $B \in K$ satisfies $\mathcal{O} \turingeq \mathcal{O}^B \turingeq B \oplus \mathcal{O}$.
\end{proof}

We can generalize \cref{gandy-kreisel basis theorem}, replacing the real $Z$ by a sequence of reals, as follows.

\begin{thm} \label{gandy-kreisel basis theorem multiple reals}
	Suppose $K$ is an uncountable $\Sigma^1_1$ class and $Z$ and $A$ are reals such that $Z \oplus \mathcal{O} \turingleq A$ and $0 \hyple (Z)_k$ for each $k \in \mathbb{N}$. Then there exists $B \in K$ such that 
	\begin{equation*}
		A \turingeq \mathcal{O}^B \turingeq B \oplus \mathcal{O}
	\end{equation*}
	and $(Z)_k \hypnleq B$ for all $k$.
\end{thm}
\begin{proof}
	The proof of \cref{gandy-kreisel basis theorem} may be adapted by replacing Stage $n=3^{b+1} \cdot 5^e \cdot 7^f$ with $n=3^{b+1} \cdot 5^e \cdot 7^f \cdot 11^k$ and replacing therein $Z$ with $(Z)_k$.
\end{proof}

\section{Posner-Robinson for Turing Degrees of Hyperjumps}

\begin{thm}[Posner-Robinson for Turing Degrees of Hyperjumps] \label{posner-robinson for hyperjumps}
	Suppose $Z$ and $A$ are reals such that $Z \oplus \mathcal{O} \turingleq A$ and $0 \hyple Z$. Then there exists $B$ such that
	\begin{equation*}
		A \turingeq \mathcal{O}^B \turingeq B \oplus Z \turingeq B \oplus \mathcal{O}.
	\end{equation*}
\end{thm}

\cref{posner-robinson for hyperjumps} is essentially due to Slaman \cite{slaman2018email}. The rest of this section is devoted to a proof of \cref{posner-robinson for hyperjumps}. The key to the proof is a forcing notion known as Kumabe-Slaman forcing, which was originally introduced in \cite{shore1999defining}.

\subsection{Kumabe-Slaman Forcing}

In order to prove \cref{posner-robinson for hyperjumps}, we will use Turing functionals and an associated notion of forcing to construct the desired $B$.

\begin{definition}[Turing Functionals] \cite{shore1999defining, reimann2018effective}
	A \textbf{Turing functional} $\Phi$ is a set of triples $(x,y,\sigma) \in \mathbb{N} \times \{0,1\} \times \{0,1\}^\ast$ (called \textbf{computations in $\Phi$}) such that if $(x,y_1,\sigma_1), (x,y_2,\sigma_2) \in \Phi$ and $\sigma_1$ and $\sigma_2$ are compatible, then $y_1 = y_2$ and $\sigma_1 = \sigma_2$. 
	
	A Turing functional $\Phi$ is \textbf{use-monotone} if:
	\begin{enumerate}[(i)]
		\item For all $(x_1,y_1,\sigma_1)$ and $(x_2,y_2,\sigma_2)$ are elements of $\Phi$ and $\sigma_1 \subset \sigma_2$, then $x_1 < x_2$. 
		
		\item For all $x_1$ and $(x_2,y_2,\sigma_2) \in \Phi$ where $x_2 > x_1$, then there are $y_1$ and $\sigma_1$ such that $\sigma_1 \subseteq \sigma_2$ and $(x_1,y_1,\sigma_1) \in \Phi$.
	\end{enumerate}
\end{definition}

\begin{remark}
	Despite the terminology, a Turing functional $\Phi$ is not assumed to be recursive or even recursively enumerable.
\end{remark}

\begin{definition}[Computations along a Real] \cite{shore1999defining, reimann2018effective}
	Suppose $\Phi$ is a Turing functional and $X \in \cantor$. Then $(x,y,\sigma) \in \Phi$ is a \textbf{computation along $X$} if $\sigma \subset X$, in which case we write $\Phi(X)(x)=y$. If for every $x \in \mathbb{N}$ there exists $y \in \{0,1\}$ and $\sigma \subset X$ such that $(x,y,\sigma) \in \Phi$, then $\Phi(X)$ defines an element of $\cantor$ (otherwise it is a partial function). 
\end{definition}

\begin{lem} \label{turing functional of real is reducible to turing functional and real}
	Suppose $\Phi$ is a Turing functional, $X \in \cantor$, and $\Phi(X) \in \cantor$. Then 
	\begin{equation*}
		\Phi(X) \turingleq \Phi \oplus X.
	\end{equation*}
\end{lem}
\begin{proof}
	Obvious from the definition of $\Phi(X)$.
\end{proof}

\begin{definition}[Kumabe-Slaman Forcing] \cite{shore1999defining, reimann2018effective}
	Define the poset $(\mathbb{P},\leq)$ as follows:
	\begin{enumerate}[(i)]
		\item Elements of $\mathbb{P}$ are pairs $(\Phi,\mathbf{X})$ where $\Phi$ is a finite use-monotone Turing functional and $\mathbf{X}$ is a finite subset of $\cantor$. 
		
		\item If $p = (\Phi_p, \mathbf{X}_p)$ and $q = (\Phi_q, \mathbf{X}_q)$ are in $\mathbb{P}$, then $p \leq q$ if
		\begin{enumerate}[(a)]
			\item $\Phi_p \subseteq \Phi_q$ and for all $(x_q,y_q,\sigma_q) \in \Phi_q \setminus \Phi_p$ and all $(x_p,y_p,\sigma_p) \in \Phi_p$, the length of $\sigma_q$ is greater than the length of $\sigma_p$.
			\item $\mathbf{X}_p \subseteq \mathbf{X}_q$.
			\item For every $x$, $y$, and $X \in \mathbf{X}$, if $\Phi_q(X)(x) = y$, then $\Phi_p(X)(x) = y$.
		\end{enumerate}		
		In other words, a stronger condition than $p$ can add longer computations to $\Phi_p$, provided they don't apply to any element of $\mathbf{X}_p$.
	\end{enumerate}
\end{definition}

In the remainder of \S 3, we will be discussing Kumabe-Slaman forcing over countable $\omega$-models of \ZFC\footnote{Here \ZFC\ denotes Zermelo-Fraenkel Set Theory with the Axiom of Choice. However, for the purposes of this paper, our $\omega$-models need not satisfy \ZFC\ but only a small subsystem of \ZFC\ or actually of second-order arithmetic.}. Unlike in the forcing constructions in axiomatic set theory, it will be important here that the countable ground model $M$ is \emph{not} well-founded. We now introduce some conventions for discussing such models.

Let $M$ be a countable non-well-founded $\omega$-model of \ZFC. Let $\theta(x_1,\ldots,x_n)$ be a sentence in the language of \ZFC\ with parameters $x_1,\ldots,x_n$ from $M$. We write $\theta^M(x_1,\ldots,x_m)$ or $M \vDash \theta(x_1,\ldots,x_n)$ to mean that $\theta(x_1,\ldots,x_n)$ holds in $M$. In particular, $x_1 \in^M x_2$ means that $M \vDash x_1 \in x_2$, etc. We tacitly identify the natural number system of $M$ with the standard natural number system, the reals of $M$ with standard reals, etc. In particular, let $\mathbb{P}^M$ be the set of pairs $(\Phi,X)$ such that $M \vDash \text{``$(\Phi,X)$ is a Kumabe-Slaman forcing condition''}$. In this case, $\Phi$ is identified with a finite Turing functional, $X$ is identified with a finite set of reals belonging to $M$, etc., so $(\Phi,X)$ actually \emph{is} a Kumabe-Slaman forcing condition.

The key property of Kumabe-Slaman Forcing is the following:

\begin{lem}\cite[based on Lemma 3.10, pg. 23]{reimann2018effective} 
\label{kumabe-slaman lemma}
	Suppose $M$ is an $\omega$-model of $\mathsf{ZFC}$, $D \in M$ is dense in $\mathbb{P}^M$, and $X_1,\ldots, X_n \in \cantor$. Then for any $p \in \mathbb{P}^M$, there exists $q \geq p$ such that $q \in D$ and $\Phi_q$ does not add any new computations along any $X_k$. 
\end{lem}
\begin{proof}
	Suppose $p = (\Phi_p, \mathbf{X}_p) \in \mathbb{P}^M$. Say that an $n$-tuple of strings $\vec{\tau}$ is \emph{essential} for $(p,D)$ if $q>p$ and $q\in D$ implies the existence of $(x,y,\sigma) \in \Phi_q \setminus \Phi_p$ such that $\sigma$ is compatible with some component of $\vec{\tau}$, i.e., any extension of $p$ in $D$ adds a computation along a string compatible with a component of $\vec{\tau}$. Being essential for $(p,D)$ is definable in $M$.
	
	Define
	\begin{equation*}
		T_n(p,D) \coloneq \{ \vec{\tau} \in (\{0,1\}^\ast)^n \mid \text{$\vec{\tau}$ is essential for $(p,D)$ and $|\tau_1| = \cdots = |\tau_n|$}\}.
	\end{equation*}
	Being essential for $(p,D)$ is closed under taking (component-wise) initial segments, so $T_n(p,D)$ is a finitely branching tree in $M$.
	
	Suppose for the sake of a contradiction that for every $q>p$, either $q \notin D$ or else $q$ adds a new computation along some $X_k$. We claim that $\langle X_1 \restrict m,\ldots, X_n \restrict m\rangle$ is essential for $(p,D)$ for all $m \in \mathbb{N}$. Given $q > p$ with $q \in D$, by hypothesis there is some computation $(x,y,\sigma) \in \Phi_q \setminus \Phi_p$ along some $X_k$. This means that $\sigma \subset X_k$ (outside of $M$), so $\sigma$ is compatible with $X_k \restrict m$. 
	
	This shows that $T_n(p,D)$ is infinite. As $M$ is a model of $\mathsf{ZFC}$, it follows that $T_n(p,D)$ has a path through it. The requirement that the components of any element of $T_n(p,D)$ are of the same length implies that such a path is of the form $(Y_1,\ldots,Y_n)$ for $Y_1,\ldots,Y_n \in M \cap \cantor$. 
	
	Consider $p_1 = (\Phi_p, \mathbf{X}_p \cup \{ Y_1,\ldots,Y_n\})$. Suppose $q \geq p_1$ and $q \in D$. Each $n$-tuple $\langle Y_1 \restrict m, \ldots, Y_n \restrict m\rangle$ is essential for $(p,D)$ for each $m$, so there exists $(x_m,y_m,\sigma_m) \in \Phi_q \setminus \Phi_p$ such that $\sigma_m$ is compatible with $Y_k \restrict m$ for some $k$. As $\Phi_q$ is finite, letting $m$ be sufficiently large shows that there is $(x,y,\sigma) \in \Phi_q \setminus \Phi_p$ for which $\sigma$ is an initial segment of $Y_k$ for some $k$. However, this is not possible since $q \leq p_1$ implies $Y_k \in \mathbf{X}_q$. This provides the needed contradiction.
\end{proof}

Suppose $G$ is an $M$-generic filter for $\mathbb{P}^M$. Then for every $X$
\begin{equation*}
	X \subseteq^M \mathbb{N} \iff \text{there is $p \in G$ with $X \in \mathbf{X}_p$}
\end{equation*}
since for any $X \subseteq^M \mathbb{N}$, the set $\{ p \in \mathbb{P}^M \mid (\emptyset, \{X\}) \leq p\}^M$ is a dense open subset of $\mathbb{P}^M$ in $M$. Thus, the essential parts of an $M$-generic filter $G$ are the Turing functionals $\Phi_p$ for $p \in G$.

\begin{definition}
	A Turing functional $\Phi$ is \textbf{$M$-generic for $\mathbb{P}^M$} if and only if there exists an $M$-generic filter $G$ such that
	\begin{equation*}
		(x,y,\sigma) \in \Phi \iff \text{there exists $p \in G$ such that $(x,y,\sigma) \in^M \Phi_p$}.
	\end{equation*}
\end{definition}

$\Phi$ may be identified with an element $(\dot{\Phi})_G$ in $M[G]$, where
\begin{equation*}
	M \vDash \dot{\Phi} = \{ (p,\dot{c}) \mid p \in \mathbb{P}^M \wedge c \in \Phi_p\}
\end{equation*}
and $\dot{c}$ is a \emph{canonical name} for $c \in M$, defined by transfinite recursion in $M$ to be the unique element in $M$ for which
\begin{equation*}
	M \vDash \dot{c} = \mathbb{P}^M \times \{ \dot{b} \mid b \in c \}.
\end{equation*}

By defining an $M$-generic Turing functional $\Phi$ for $\mathbb{P}^M$ by means of approximations, \cref{kumabe-slaman lemma,,hyperarithmetical iff in every omega model} allow us to meet dense sets without affecting $\Phi(Z)$, which can then be arranged independently.

\subsection{Proof of Posner-Robinson for Hyperjumps}

Now we proceed with the proof of \cref{posner-robinson for hyperjumps}:

\begin{lem} \label{gandy-kreisel basis theorem omega-model avoiding sequence of nonhyp reals} \label{hyperarithmetical iff in every omega model}
	Suppose $Z$ and $A$ are reals such that $Z \oplus \mathcal{O} \turingleq A$ and $0 \hyple Z$. Then there exists a (code for a) countable $\omega$-model $M$ of $\mathsf{ZFC}$ such that $\mathcal{O}^M \turingeq A$ and $Z \notin M$.
\end{lem}
\begin{proof}
	The set of codes for countable $\omega$-models of $\mathsf{ZFC}$ is $\Sigma^1_1$, so the existence of a code of such an $M$ follows from \cref{gandy-kreisel basis theorem}.
\end{proof}

\begin{proof}[Proof of \cref{posner-robinson for hyperjumps}:] 
	The main idea of the proof is due to Slaman \cite{slaman2018email}.
	
	We shall construct an $M$-generic Turing functional $\Phi$ with $B = \Phi$ the desired real. Assume without loss of generality that no initial segment of $Z$ is an initial segment of $\mathcal{O}$. By arranging for $\Phi(Z) \in \cantor$ and $\Phi(Z) = \mathcal{O}^\Phi$ and $\Phi(\mathcal{O}) = A$, this will complete the proof. 
	
	By \cref{gandy-kreisel basis theorem omega-model avoiding sequence of nonhyp reals}, there exists a countable $\omega$-model $M$ of $\mathsf{ZFC}$ such that $\mathcal{O},Z \notin M$ and $\mathcal{O}^M \turingeq A$. Without loss of generality, $M = \langle \omega, E\rangle$.
	
	Let $D_0,D_1,D_2,\ldots$ be an enumeration, recursive in $A$, of the dense open subsets of $\mathbb{P}^M$ in $M$ ($M$ is countable and $\mathcal{O}^M \turingeq A$, so this is possible). To construct our $M$-generic $\Phi$, we approximate it by finite initial segments
	\begin{equation*}
		p_0 \leq p_1 \leq \cdots \leq p_n \leq \cdots.
	\end{equation*}
	During our construction, we alternate between meeting dense sets, arranging for $\Phi(\mathcal{O}) = A$, and arranging for $\Phi(Z) \turingeq \mathcal{O}^\Phi$.
	
	\begin{description}
		\item[Stage $n=0$:] Define $p_0 \coloneq (\emptyset,\emptyset)$. 
		
		\item[Stage $n=2^m$:] Suppose $p_{n-1}$ has been constructed. By \cref{kumabe-slaman lemma}, there exists $q \in D_n$ extending $p_{n-1}$ which does not add any new computations along $Z$ or $\mathcal{O}$. Let $p_n$ be the least such condition.
		
		\item[Stage $n = 2^m \cdot 3$:] We extend $p_{n-1}$ to $p_{n}$ by adding $(m,A(m),\sigma)$ where $\sigma \subset \mathcal{O}$ is a sufficiently long initial segment of $\mathcal{O}$ (i.e., the shortest initial segment of $\mathcal{O}$ which is longer than any existing strings in elements of $\Phi_{p_{n-1}}$).
		
		\item[Stage $n=2^m \cdot 5$:] Suppose $p_{n-1}$ has been constructed. By construction, there is no $y$ and $\sigma \subset Z$ such that $(m,y,\sigma) \in \Phi_{p_{n-1}}$. Now proceed as follows:
		\begin{description}
			\item[Substage 1:] Consider the set $D$ (in $M$) containig all $q \in \mathbb{P}^M$ such that one of the following conditions hold:
			\begin{enumerate}[(i)]
				\item $q \forces ( \text{$m$ encodes a ${{\Phi}}$-recursive linear order on $\omega$} \wedge m \in \mathcal{O}^{{{\Phi}}} \wedge \exists \alpha ~( \alpha \in \Ord^M \wedge |m| = \alpha))$,
				\item $q \forces (\text{$m$ encodes a ${{\Phi}}$-recursive linear order on $\omega$} \wedge m \notin \mathcal{O}^{{{\Phi}}})$, or
	\item $q \forces \neg (\text{$m$ encodes a ${{\Phi}}$-recursive linear order on $\omega$})$.
			\end{enumerate}
			$D$ is dense. By \cref{kumabe-slaman lemma}, there exists $q\in D$ extending $p_{n-1}$ which does not add any new computations along $Z$ or $\mathcal{O}$. Let $q$ be minimal with that property.
	
			\item[Substage 2:] Extend $q$ to $p_n$ by adding $(m,y,\sigma)$, where $\sigma \subset Z$ is a sufficiently long initial segment of $Z$ (i.e., the shortest initial segment of $Z$ which is longer than any existing strings in elements of $\Phi_q$) and $y$ depends on the following cases:
			\begin{description}
				\item[Case 1:] If $q \forces (\text{$m$ encodes a ${{\Phi}}$-recursive linear order on $\omega$} \wedge m \in \mathcal{O}^{{{\Phi}}} \wedge \exists \alpha ~( \alpha \in \Ord^M \wedge |m| = \alpha))$, then we break into two subcases:
				\begin{description}
					\item[Case 1a:] If $\alpha$ is in the standard part of $\Ord^M$, then $\alpha$ is actually an ordinal and $m$ \emph{does} encode a ${{\Phi}}$-recursive linear order on $\omega$. Thus, set $y \coloneq 1$.
	
					\item[Case 1b:] If $\alpha$ is not in the standard part of $\Ord^M$, then $\alpha$ is not actually well-ordered (it is only well-ordered when viewed in $M$) so $m$ \emph{does not} encode a ${{\Phi}}$-recursive linear order on $\omega$.  Thus, set $y \coloneq 0$.
				\end{description}
	
				\item[Case 2:] If $q \forces (\text{$m$ encodes a ${{\Phi}}$-recursive linear order on $\omega$} \wedge m \notin \mathcal{O}^{{{\Phi}}})$, then $m$ cannot encode a ${{\Phi}}$-recursive well-ordering of $\omega$. Thus, set $y \coloneq 0$.
	
				\item[Case 3:] If $q \forces \neg (\text{$m$ encodes a ${{\Phi}}$-recursive linear order on $\omega$})$, then set $y \coloneq 0$.
			\end{description}
		\end{description}
	
		\item[All Other Stages $n$:] Let $p_n = p_{n-1}$.
	\end{description}
	
	Define $\Phi$ to be the unique set such that
	\begin{equation*}
		(x,y,\sigma) \in \Phi \iff \text{there exists $n \in \mathbb{N}$ such that $(x,y,\sigma) \in^M \Phi_{p_n}$}.
	\end{equation*}
	Thanks to Stages $n=2^m$ and the observation that $G \coloneq \{ q \in \mathbb{P}^M \mid \exists n ~ (q \leq p_n) \}$ is consequently $M$-generic, $\Phi$ is an $M$-generic Turing functional. Thanks to Stages $n=2^m \cdot 3$, $\Phi(\mathcal{O}) = A$. Thanks to Stages $n=2^m \cdot 5$, $\Phi(Z) = \mathcal{O}^\Phi$.
	
	We also note that in the above construction of $\Phi$, (assuming $p_{n-1}$ is given)\ldots
	\begin{description}
		\item[] \ldots Stage $n=2^m$ is recursive in $\mathcal{O}^M \oplus Z \oplus \mathcal{O} \turingleq A$,
		\item[] \ldots Stage $n=2^m \cdot 3$ is recursive in $\mathcal{O} \turingleq A$, 
		\item[] \ldots Stage $n=2^m \cdot 5$ (Substage 1) is recursive in $\mathcal{O} \oplus \mathcal{O}^M \oplus Z \turingleq A$, 
		\item[] \ldots Stage $n=2^m \cdot 5$ (Substage 2) is recursive in $Z \turingleq A$, 
		and 
		\item[] \ldots Stage $n$ (for all other $n$) is recursive.
	\end{description}
	Thus,
	\begin{equation*}
		\Phi \turingleq M \oplus (\mathcal{O}^M \oplus Z) \oplus A \turingleq A.
	\end{equation*}
	
	Applying \cref{turing functional of real is reducible to turing functional and real} we find
	\begin{equation*}
		A = \Phi(\mathcal{O}) \turingleq \mathcal{O} \oplus \Phi \turingleq \mathcal{O}^\Phi \turingeq \Phi(Z) \turingleq Z \oplus \Phi \turingleq Z \oplus A \turingeq A
	\end{equation*}
	so we have Turing equivalence throughout. $B= \Phi$ is hence the desired real.
\end{proof}

\cref{posner-robinson for hyperjumps} can be generalized, replacing the real $Z$ by a sequence of reals.

\begin{thm} \label{posner-robinson for hyperjumps multiple reals}
	Suppose $Z$ and $A$ are reals such that $Z \oplus \mathcal{O} \turingleq A$ and $0 \hyple (Z)_k$ for every $k \in \mathbb{N}$. Then there exists $B$ such that for every $k \in \mathbb{N}$
	\begin{equation*}
		A \turingeq \mathcal{O}^B \turingeq B \oplus (Z)_k \turingeq B \oplus \mathcal{O}.
	\end{equation*}
\end{thm}
\begin{proof}
	The proof of \cref{posner-robinson for hyperjumps} may be adapted by making the following adjustments. First, we replace the use of \cref{gandy-kreisel basis theorem omega-model avoiding sequence of nonhyp reals} with the following lemma:
	
	\begin{lem}
		Suppose $Z$ and $A$ are reals such that $Z \oplus \mathcal{O} \turingleq A$ and $0 \hyple (Z)_k$ for every $k \in \mathbb{N}$. Then there exists a (code for a) countable $\omega$-model $M$ of $\mathsf{ZFC}$ such that $\mathcal{O}^M \turingeq A$ and $(Z)_k \notin M$ for every $k \in \mathbb{N}$.
	\end{lem}
	\begin{proof}
		Replace the usage of \cref{gandy-kreisel basis theorem} in the proof of \cref{gandy-kreisel basis theorem omega-model avoiding sequence of nonhyp reals} with \cref{gandy-kreisel basis theorem multiple reals}.
	\end{proof}
	
	This yields a (code for a) countable $\omega$-model $M$ of $\mathsf{ZFC}$ such that $\mathcal{O}, (Z)_0,(Z)_1,\ldots \notin M$ and $\mathcal{O}^M \turingeq A$. We assume without loss of generality that $\mathcal{O} \neq (Z)_k$ for each $k$. 
	
	The adjustments to the construction are the following:
	\begin{itemize}
		\item In Stages $n=2^m$ and $n=2^m\cdot 3$, we avoid adding new computations along $(Z)_0,\ldots, (Z)_n$ and $\mathcal{O}$. 
		\item Replace Stage $n=2^m \cdot 5$ with Stages $n=2^m \cdot 5^{k+1}$, and at the beginning of Stage $n=2^m \cdot 5^{k+1}$, first check if there exists $y$ and $\sigma \subset (Z)_k$ such that $(m,y,\sigma) \in \Phi_{p_{n-1}}$. If such a $y$ and $\sigma$ are found, do nothing and proceed to the next stage. Otherwise, proceed as in Stage $n=2^m\cdot 5$ of the proof of \cref{posner-robinson for hyperjumps}, with the same adjustment of avoiding adding new computations along $(Z)_0,\ldots, (Z)_n$ and $\mathcal{O}$ as above.
	\end{itemize}
	
	Note that it is no longer necessarily the case that $\Phi((Z)_k) = \mathcal{O}^\Phi$ for every $k \in \mathbb{N}$, as early stages may have added computations to $\Phi$ which make $\Phi((Z)_k)$ disagree with $\mathcal{O}^\Phi$. However, after Stage $k$, no other stages add new computations along $(Z)_k$ except for those purposely added (i.e., in Stages $n=2^m \cdot 5^{k+1}$). It follows that $\Phi((Z)_k)$ and $\mathcal{O}^\Phi$ differ only on a finite set of indices, so $\Phi((Z)_k) \turingeq \mathcal{O}^\Phi$.
	
	In the resulting construction of $\Phi$, (assuming $p_{n-1}$ is given)
	\begin{description}
		\item[] \ldots Stage $n=2^m$ is recursive in $\mathcal{O}^M \oplus \bigoplus_{i=0}^n{(Z)_i} \oplus \mathcal{O} \turingleq A$,
		\item[] \ldots Stage $n=2^m \cdot 3$ is recursive in $\mathcal{O} \turingleq A$, 
		\item[] \ldots Stage $n=2^m \cdot 5^{k+1}$ (Substage 1) is recursive in $\mathcal{O} \oplus \mathcal{O}^M \oplus \bigoplus_{i=0}^n{(Z)_i} \turingleq A$, 
		\item[] \ldots Stage $n=2^m \cdot 5^{k+1}$ (Substage 2) is recursive in $(Z)_k \turingleq A$, 
		and 
		\item[] \ldots Stage $n$ (for all other $n$) is recursive.
	\end{description}
	Thus,
	\begin{equation*}
		\Phi \turingleq M \oplus (\mathcal{O}^M \oplus Z) \oplus A \turingeq A.
	\end{equation*}
	The proof concludes as in the proof of \cref{posner-robinson for hyperjumps}.
\end{proof}

\section{Open Problems}

In light of \cref{gandy-kreisel basis theorem,,posner-robinson for hyperjumps}, it is natural to ask whether they can be combined into one theorem. In other words, for which uncountable $\Sigma^1_1$ classes $K \subseteq \cantor$ do the following properties hold?

\begin{property} \label{posner-robinson property}
	Suppose $Z$ and $A$ are reals such that $Z \oplus \mathcal{O} \turingleq A$ and $0 \hyple (Z)_k$ for every $k \in \mathbb{N}$. Then there exists $B \in K$ such that
	\begin{equation*}
		A \turingeq \mathcal{O}^B \turingeq B \oplus Z \turingeq B \oplus \mathcal{O}.
	\end{equation*}
\end{property}

\begin{property} \label{countable posner-robinson property}
	Suppose $Z$ and $A$ are reals such that $Z \oplus \mathcal{O} \turingleq A$ and $0 \hyple (Z)_k$ for every $k \in \mathbb{N}$. Then there exists $B \in K$ such that for every $k$
	\begin{equation*}
		A \turingeq \mathcal{O}^B \turingeq B \oplus (Z)_k \turingeq B \oplus \mathcal{O}.
	\end{equation*}
\end{property}

The following theorem answers some special cases of this problem.

\begin{thm} \label{posner-robinson for hyperjumps in upward closed sigma11 classes}
	Let $L_\turing = \{X \mid \mathcal{O}^X \turingeq X \oplus \mathcal{O}\}$. Suppose $K$ is an uncountable $\Sigma^1_1$ class which is Turing degree upward closed in $L_\turing$, i.e., whenever $X,Y \in L_\turing$, $X \in K$, and $X \turingleq Y$, then there is $Y_0 \in K$ such that $Y \turingeq Y_0$. Then $K$ has \cref{{posner-robinson property},,countable posner-robinson property}.
\end{thm}
\begin{proof}
	This theorem is analogous to \cite[Lemma 3.3]{jananthan2020pseudojump}. By \cref{gandy-kreisel basis theorem}, let $C$ be such that 
	\begin{equation} \label{posner-robinson for hyperjumps c}
		A \turingeq \mathcal{O}^C \turingeq C \oplus Z \turingeq C \oplus \mathcal{O}. \tag*{($\ast$)}
	\end{equation}
	\cref{gandy-kreisel basis theorem}, relativized to $C$, yields $B_0 \in K$ such that
	\begin{equation} \label{gandy basis theorem relative to c}
		\mathcal{O}^C \turingeq \mathcal{O}^{B_0 \oplus C} \turingeq B_0 \oplus \mathcal{O}^C. \tag*{($\dagger$)}
	\end{equation}
	Combining \ref{gandy basis theorem relative to c} and \ref{posner-robinson for hyperjumps c} shows that $\mathcal{O}^{B_0 \oplus C} \turingeq B_0 \oplus C \oplus \mathcal{O}$. As $B_0 \turingleq B_0 \oplus C$, there is $B \in K$ such that $B \turingeq B_0 \oplus C$ by hypothesis. In particular,
	\begin{equation*}
		\mathcal{O}^C \turingeq \mathcal{O}^B \turingeq B \oplus \mathcal{O}.
	\end{equation*}
	Moreover, in combination with \ref{posner-robinson for hyperjumps c},
	\begin{equation*}
		A \turingeq \mathcal{O}^B \turingeq B \oplus \mathcal{O} \turingeq B \oplus Z.
	\end{equation*}
	This shows that $K$ has \cref{posner-robinson property}.

	To show that $K$ has \cref{countable posner-robinson property}, repeat the above argument using \cref{gandy-kreisel basis theorem multiple reals} instead of \cref{gandy-kreisel basis theorem}.
\end{proof}

\begin{remark}
	The proof of \cref{posner-robinson for hyperjumps in upward closed sigma11 classes} is easily adapted to prove the same result with $L_\turing = \{ X \mid \mathcal{O}^X \turingeq X \oplus \mathcal{O} \}$ replaced by $L_ \hyp = \{ X \mid \mathcal{O}^X \hypeq X \oplus \mathcal{O}\}$.
\end{remark}

The hyperarithmetical analog of the Pseudojump Inversion Theorem \cite[Theorem 2.1, pg. 601]{jockusch1983pseudojump} also remains open. Namely, suppose $V_e^X$ is an effective enumeration of the $\Pi^{1,X}_1$ predicates, uniformly in $X$. Define the \textbf{$e$-th pseudo-hyperjump} by 
\begin{equation*}
	\hj_e(X) \coloneq X \oplus V_e^X.
\end{equation*}
Does the following result hold?

\begin{conj} \label{pseudo-hyperjump inversion}
	Suppose $e \in \mathbb{N}$ and $A$ is a real such that $\mathcal{O} \turingleq A$. Then there exists $B$ such that
	\begin{equation} \label{pseudo-hyperjump inversion equation}
		A \turingeq \hj_e(B) \turingeq B \oplus \mathcal{O}.
	\end{equation}
\end{conj}

Even if \cref{pseudo-hyperjump inversion} holds, this leaves open the question of characterizing the $\Sigma^1_1$ classes $K \subseteq \cantor$ with the following properties:

\begin{property}
	Suppose $e \in \mathbb{N}$ and $A$ is a real such that $\mathcal{O} \turingleq A$. Then there exists $B \in K$ such that \cref{pseudo-hyperjump inversion equation} holds.
\end{property}

\begin{property}
	Suppose $e \in \mathbb{N}$ and $Z$ and $A$ are reals such that $Z \oplus \mathcal{O} \turingleq A$ and $0 \hyple (Z)_k$ for each $k \in \mathbb{N}$. Then there exists $B \in K$ such that \cref{pseudo-hyperjump inversion equation} holds and $(Z)_k \nhypleq B$ for every $k \in \mathbb{N}$.
\end{property}

\begin{property}
	Suppose $e \in \mathbb{N}$ and $Z$ and $A$ are reals such that $Z \oplus \mathcal{O} \turingleq A$ and $0 \hyple Z$. Then there exists $B \in K$ such that
	\begin{equation*}
		A \turingeq \hj_e(B) \turingeq B \oplus Z \turingeq B \oplus \mathcal{O}.
	\end{equation*}
\end{property}

\bibliographystyle{amsplain}
\bibliography{biblio}

\end{document}